\documentclass[10pt, reqno]{amsart}
\usepackage{amsmath, amsthm, amscd, amsfonts, amssymb, graphicx, color}
\usepackage[bookmarksnumbered, colorlinks, plainpages]{hyperref}
\hypersetup{colorlinks=true,linkcolor=red, anchorcolor=green, citecolor=cyan, urlcolor=red, filecolor=magenta, pdftoolbar=true}
\usepackage{mathrsfs}
\newtheorem{theorem}{Theorem}[section]
\newtheorem{lemma}[theorem]{Lemma}

\newtheorem{corollary}[theorem]{Corollary}
\theoremstyle{definition}

\theoremstyle{remark}
\newtheorem{remark}[theorem]{Remark}
\numberwithin{equation}{section}

\begin{document}
\setcounter{page}{1}

\title[Profinite mapping class groups]{Profinite mapping class groups}

\author[I.~V.~Nikolaev]
{Igor V. Nikolaev$^1$}

\address{$^{1}$ Department of Mathematics and Computer Science, St.~John's University, 8000 Utopia Parkway,  
New York,  NY 11439, United States.}
\email{\textcolor[rgb]{0.00,0.00,0.84}{igor.v.nikolaev@gmail.com}}

%\dedicatory{In memory of I.~R.~ Shafarevich}



%\date{Received:  August 14, 2015; Revised: yyyyyy; Accepted: zzzzzz.}

\begin{abstract}
It is  proved  that the profinite completion of the mapping class group $Mod ~(g,n)$
of a surface of genus $g$ with $n$ boundary components is isomorphic to such 
of the arithmetic group  $GL_{6g-6+2n}(\mathbf{Z})$. 
We establish a relation  between the normal subgroups of  
$Mod ~(g,n)$ and the absolute Galois group $G_K$ of a number field $K$.
Using the Tits alternative, we prove the Shafarevich Conjecture saying 
 that  the  group $G_{\mathbf{Q}^{ab}}$  
of the maximal  abelian extension   of the field of rationals
is isomorphic to a free profinite group. 
 \end{abstract}

\maketitle

%**************************************************************************
\section{Introduction}
%***************************************************************************
The mapping class group $Mod ~(g,n)$ of an orientable surface $X$ of genus $g\ge 0$ with $n\ge 0$
boundary components is defined as a group of isotopy classes of the orientation and boundary-preserving 
diffeomorphisms of $X$.   Since  $Mod ~(1,0)\cong SL_2(\mathbf{Z})$ is an arithmetic group, 
one can ask if $Mod ~(g,n)$ is always arithmetic.  It is proven   to be false by  [Ivanov 1988] \cite[Theorem 1]{Iva1}.
Roughly speaking, the reason  is the Torelli  group,  which is a normal
subgroup  of $Mod ~(g,n)$ of infinite index.  This fact goes against  the Margulis Rigidity Theorem,
which says that each normal subgroup of the arithmetic group must have a finite index. 
Despite being  non-arithmetic itself,  the $Mod ~(g,n)$ can be embedded into 
the arithmetic group  $GL_{6g-6+2n}(\mathbf{Z})$ \cite{Nik1}. 
We refer the reader to [Harvey 1979]  \cite[Section 6]{Har1} for a survey of 
the  arithmetic  properties of  $Mod ~(g,n)$.

Recall that a profinite group $\widehat{G}$ is a topological group defined  by the inverse limit
%*****************************************************************************************************
\begin{equation}\label{eq1.1}
\widehat{G}:=\varprojlim G/N,
\end{equation}
%************************************************************************************************
where $G$ is a discrete group   and  $N$ ranges through the  open normal finite index subgroups of 
$G$.   It is not hard to see, that  if $G_1\cong G_2$, then $\widehat{G}_1\cong \widehat{G}_2$.  
But the converse is false in general. Recall that the groups $G_1\hookrightarrow G_2$ are called 
a {\it Grothendieck pair},  if $\widehat{G}_1\cong \widehat{G}_2$. 
Roughly speaking, such a property means that the groups $G_1$ and $G_2$ are similar 
from the viewpoint of  representation theory. The Grothendieck pairs are known  to exist, see [Platonov \& Tavgen 1986] 
\cite{PlaTav1} and  
[Bridson \& Grunewald 2004] \cite{BriGru1}. 
In this note we  show that $Mod ~(g,n)\hookrightarrow GL_{6g-6+2n}(\mathbf{Z})$
are  a Grothendieck pair. Our main result can be formulated as follows.
%*****************************************************************************************************
\begin{theorem}\label{thm1.1}
$\widehat{Mod} ~(g,n)\cong \widehat{GL}_{6g-6+2n}(\mathbf{Z}).$
\end{theorem}
%************************************************************************************************

An application of theorem \ref{thm1.1} is as follows.
Let $K$ be a number field and let $\bar K$ be its algebraic closure. 
Denote by  $G_K:=Gal~(\bar K|K)$ the absolute Galois group of the field $K$. 
Let $\mathscr{N}_{g,n}$ be a category of  normal subgroups of the mapping class group  $Mod~(g,n)$,
where  the arrows of $\mathscr{N}_{g,n}$  are  isomorphisms between such  subgroups.
Likewise, let $\mathscr{K}$ be a category of the Galois  extensions of the field $\mathbf{Q}$,  where 
the arrows of $\mathscr{K}$ are  isomorphisms between such extensions.
Consider a map $F_{g,n}$ acting by the formula
$N\mapsto \widehat{N} / \widehat{\mathbf{Z}} ~\cong  G_K$,   where $N\in \mathscr{N}_{g,n}$
and  $K\in\mathscr{K}$. 
%********************************************************************************************
\footnote{We refer the reader to Section 3 for the motivation and construction of the map $F_{g,n}$.}
%*******************************************************************************************
%*****************************************************************************************************
\begin{theorem}\label{thm1.2}
The map $F_{g,n}: \mathscr{N}_{g,n}\to \mathscr{K}$  is an injective  functor,
unless $N, N'\in \mathscr{N}_{g,n}$ are a Grothendieck pair.
Moreover, for every finite index normal subgroup $N'\subseteq N$,  there exists 
an intermediate  field $K'=F_{g,n}(N')$, such that $K\subseteq K'\subset \bar K$ and 
%***********************************************************************************
\begin{equation}\label{eq1.5}
Gal~(K'|K)\cong N/N'. 
\end{equation} 
%****************************************************************************
 \end{theorem}
%************************************************************************************************

\medskip
Recall that the {\it Tits alternative}  for the mapping class group says that every subgroup
$N\subseteq Mod~(g,n)$ contains  either (i)  an abelian subgroup of finite index  or  (ii)  a non-abelian  free group
[McCarthy 1985]  \cite{Mac1}.  One gets  from theorem \ref{thm1.2}
an analog of the Tits alternative for the absolute Galois  group $G_K$. 
%**********************************************************************************************
\begin{corollary}\label{cor1.3}
{\bf (Tits alternative for $G_K$)}
For every number field $K\in\mathscr{K}$,  there exists an intermediate  field $K\subseteq K'\subset \bar K$,
such that the absolute Galois group $G_{K'}$ is:

\medskip
(i) either a free abelian profinite group $\widehat{\mathbf{Z}}^r$ of rank $r\le 3g-3+n$, 

\smallskip
(ii) or a free non-abelian profinite group $\widehat{F}_r$ of rank $r\ge 2$.
\end{corollary}
%***********************************************************************************
%***********************************************************************************
%\begin{remark}\label{rmk1.3}
%Roughly speaking, the Tits alternative \ref{cor1.3} says that  every extension of $K$ is either 
 %(i) abelian  or (ii)  cyclotomic, or else $K$ extends to a number field $K'$ with such a property. 
 %In other words, all extensions of sufficiently high degree  are abelian.
 %Equivalently,  there are no extensions of $K$  with a non-commutative Galois group,
 %except possibly for a finite number of such extensions. 
%\end{remark}
%*********************************************************************************

Let $F_{\infty}$ be a free non-abelian  group of countable rank. 
Let $\mathbf{Q}^{ab}$ be the maximal abelian   extension of the field $\mathbf{Q}$, 
i.e. an extension of $\mathbf{Q}$  by all roots of unity (a cyclotomic extension). 
We use case  (ii) of  the Tits alternative \ref{cor1.3} 
to  prove  the following  conjecture of I. ~R. ~Shafarevich. 
%*****************************************************************************************************
\begin{corollary}\label{cor1.4}
%{\bf (Shafarevich conjecture)}
$G_{\mathbf{Q}^{ab}}\cong \widehat{F}_{\infty}$.
\end{corollary}
%************************************************************************************************
The article is organized as follows. The preliminary facts and notation are introduced
in Section 2.  The map $F_{g,n}$ is constructed in Section 3. The  results  \ref{thm1.1}-\ref{cor1.4} 
are proved in Section 4.

%**************************************************************************
\section{Preliminaries}
%***************************************************************************
This section is a brief review of the mapping class groups, profinite groups, the Grothendieck pairs 
and the absolute Galois  group of  a number field.  We refer the reader to [Farb \& Margalit 2011]  \cite{FM} and 
[Ribes \& Zalesskii 2010] \cite{RZ}, [Platonov \& Tavgen 1986] 
\cite{PlaTav1} and  [Bridson \& Grunewald 2004] \cite{BriGru1}
for  a detailed account.

%**************************************************************************
\subsection{Mapping class group}
%***************************************************************************
Let $X$ be an orientable surface  of genus $g\ge 0$ with $n\ge 0$
boundary components. The {\it mapping class group}   $Mod ~(g,n)$
is defined as the group of isotopy classes of the  orientation and boundary-preserving
diffeomorphisms of $X$.  Since $Mod ~(1,0)\cong SL_2(\mathbf{Z})$,
one can think of $Mod ~(g,n)$ as an extension of
 the modular group to the higher genus surfaces.
The group $Mod ~(g,n)$ is prominent in 
geometric topology, complex analysis and algebraic geometry.  A link to 
number theory has been established in [Grothendieck 1997] \cite{Gro1}.

%**************************************************************************
\subsubsection{Dehn twists}
%***************************************************************************
Let  $\gamma\subset X$ is a simple closed  curve and   $A=S^1\times [0,1]$ is 
the annular neighborhood of $\gamma$. The map $T_{\gamma}: A\to A$ given by the
formula
%********************************************************************************
%\begin{equation}\label{eq2.1}
$(\theta, ~t)\longmapsto (\theta+2\pi t, ~t), \quad\theta\in S^1, \quad t\in [0,1]$, 
%\end{equation}
%**************************************************************************
is called the  Dehn twist  around $\gamma$.  It is easy to see, that  
$T_{\gamma}$ is an  infinite order element of  the group $Mod ~(g,n)$.

%**************************************************************************
\subsubsection{Pseudo-Anosov diffeomorphisms}
%***************************************************************************
Let $\mathcal{F}$ be a measured foliation on $X$
\cite[Section 0.3.2]{N}.  An element $\varphi\in Mod ~(g,n)$
is called  pseudo-Anosov,   if there exist a pair consisting of the stable $\mathcal{F}_s$
and unstable $\mathcal{F}_u$ mutually orthogonal measured foliations,
such that 
%********************************************************************************
%\begin{equation}\label{eq2.2}
$\varphi(\mathcal{F}_s)={1\over\lambda_{\varphi}}\mathcal{F}_s$
and $\varphi(\mathcal{F}_u)=\lambda_{\varphi}\mathcal{F}_u$,
%\end{equation}
%**************************************************************************
where $\lambda_{\varphi}>1$ is called a dilatation of $\varphi$.

%**************************************************************************
\subsubsection{Subgroups of $Mod~(g,n)$}
%***************************************************************************
The Dehn twist $T_{\gamma}$  is a generator of the abelian subgroup  of 
 $Mod ~(g,n)$.  Since there are at most $3g-3+n$  distinct
 simple closed curves on $X$, the rank  of the corresponding 
 subgroup $G\cong \mathbf{Z}^r\subset Mod~(g,n)$ is $r\le 3g-3+n$.  
 To the contrast, any collection $\{\varphi_i\}_{i=1}^{\infty}$
  of the pseudo-Anosov diffeomorphisms $\varphi_i$ with the pairwise 
  distinct measured foliations $\left(\mathcal{F}_s^{(i)}, \mathcal{F}_u^{(i)}\right)$
generates a free non-abelian  subgroup $F_{\infty}\subset Mod~(g,n)$ of countable rank.

%**************************************************************************
\subsubsection{Tits alternative}
%***************************************************************************
%***************************************************************************
\begin{theorem}\label{Tits}
{\bf ([McCarthy 1985]  \cite[Theorem A]{Mac1})}
 Every subgroup $G\subseteq Mod ~(g,n)$ 
  satisfies the Tits alternative:
 
 \medskip
 (i) either $G$ contains an abelian subgroup of
finite index,
 
 \smallskip
 (ii)  or $G$ contains a non-abelian  free group.
\end{theorem}
%*******************************************************************************

%**************************************************************************
\subsubsection{Linear representation of $Mod~(g,n)$}
%***************************************************************************
%****************************************************************************
\begin{theorem}\label{thm2.0}
There exists an embedding of  $Mod ~(g,n)\hookrightarrow GL_{6g-6+2n}(\mathbf{Z})$.
\end{theorem}
%***************************************************************************
\begin{proof}
The proof is an adaption of the argument of 
 \cite{Nik1} to the surfaces with  $n$ boundary components.   
\end{proof}

%**************************************************************************
\subsection{Profinite groups}
%***************************************************************************
Let $\mathscr{C}$ be a non-empty class of finite groups. 
A {\it pro-$\mathscr{C}$ group}  $\widehat{G}$ is an inverse limit 
%*****************************************************************************************************
\begin{equation}\label{eq2.3}
\widehat{G}:=\varprojlim G_i
\end{equation}
%************************************************************************************************
of surjective inverse system of groups $G_i\in\mathscr{C}$, where each
$G_i$ is endowed with the discrete topology. 
The pro-$\mathscr{C}$ group $\widehat{G}$ is a topological group
in the product topology $\prod G_i$. Such a group is compact and totally 
disconnected.

In what follows, we let  $\mathscr{C}$ be a class of finite groups. In this case, 
$\widehat{G}$  is called a {\it profinite group}. 
If $G$ is a discrete group, one can define $G_i=G/N_i$,
where $N_i$ ranges through the  normal subgroups of 
$G$ of finite index.   It is easy to see, that 
formulas (\ref{eq1.1}) and (\ref{eq2.3}) are equivalent.

%**************************************************************************
\subsection{Grothendieck pairs}
%***************************************************************************
Let $G_1$ and $G_2$ be discrete groups, such that $G_1\cong G_2$. 
In this case, their profinite completions are isomorphic, i.e. 
$\widehat{G}_1\cong \widehat{G}_2$. The groups $G_1$ and $G_2$ are called  Grothendieck 
rigid when  the converse is true, i.e. $\widehat{G}_1\cong \widehat{G}_2$ implies $G_1\cong G_2$.
Not all discrete groups are Grothendieck rigid and an inclusion of groups $G_1\hookrightarrow G_2$
is called  the {\it Grothendieck pair}, if  $\widehat{G}_1\cong \widehat{G}_2$.

%**************************************************************************
\subsection{Absolute Galois group}
%***************************************************************************
Let $K$ be a number field. Suppose that $\bar K$ is the separable algebraic
closure of $K$, i.e. the union of all separable extensions of $K$. 
By the absolute Galois group
%****************************************************************************
\begin{equation}\label{eq2.4}
 G_K:=Gal~(\bar K|K)
\end{equation}
%***************************************************************************
we understand the group of  automorphisms of $\bar K$ fixing the field $K$.  
The $G_K$ is a profinite group (\ref{eq2.3}) with $G_i=G_K/G_{K_i}$,
where $G_{K_i}$ is a closed normal subgroup of $G_K$ of corresponding to an 
intermediate number field  $K\subset K_i\subset\bar K$ [Krull 1928] \cite{Kru1}.

%**************************************************************************
\subsubsection{Rigidity of $G_K$}
%***************************************************************************
The number field $K$ is defined up to an isomorphism by the group $G_K$, 
i.e. $K\cong K'$ if and only if $G_K\cong G_{K'}$.

%**************************************************************************
\subsubsection{Shafarevich conjecture}
%***************************************************************************
Let $F_{\infty}$ be a free non-abelian  group of countable rank. 
Let $\mathbf{Q}^{ab}$ be the maximal abelian   extension of the field $\mathbf{Q}$, 
i.e. an extension of $\mathbf{Q}$  by all roots of unity (a cyclotomic extension). 
The Shafarevich conjecture asserts  that:
%****************************************************************************
\begin{equation}\label{eq2.5}
G_{\mathbf{Q}^{ab}}\cong \widehat{F}_{\infty}.
\end{equation}
%***************************************************************************

%**************************************************************************
\section{Map  $F_{g,n}$}
%***************************************************************************
Let $K$ be a number field and let $\bar K$ be its algebraic closure. 
Denote by  $G_K:=Gal~(\bar K|K)$ the absolute Galois group of the field $K$.

%**************************************************************************
\subsection{Short exact sequence for $G_K$}
%***************************************************************************
Fix an embedding $\bar K\subset \mathbf{C}$ and
consider  a natural inclusion of the algebraic groups 
$GL_{6g-6+2n}(K)\hookrightarrow GL_{6g-6+2n}(\mathbf{C})$.
Let 
 %*****************************************************************************************************
\begin{equation}\label{eq3.13}
1\to \pi_1^{et}(GL_{6g-6+2n}(\mathbf{C}))\to \pi_1^{et}(GL_{6g-6+2n}(K))\to  G_K\to 1
\end{equation}
%************************************************************************************************
be a short exact sequence of the \'etale fundamental groups
corresponding to the  map  $GL_{6g-6+2n}(K)\hookrightarrow GL_{6g-6+2n}(\mathbf{C})$. 
It is known that  $\pi_1^{et}(GL_{6g-6+2n}(\mathbf{C}))\cong \widehat{\pi}_1({GL_{6g-6+2n}(\mathbf{C})})$,
where $\pi_1(GL_{6g-6+2n}(\mathbf{C}))$ is the usual fundamental of the variety 
$GL_{6g-6+2n}(\mathbf{C})$.   Since  $\pi_1(GL_{6g-6+2n}(\mathbf{C}))\cong \mathbf{Z}$, 
one gets an isomorphism:
%***********************************************************************************
\begin{equation}
 \pi_1^{et}(GL_{6g-6+2n}(\mathbf{C}))\cong\widehat{\mathbf{Z}}.
 \end{equation}
 %*********************************************************************************************
 Since $GL_{6g-6+2n}(K)$ is an algebraic group,  we have an isomorphism:
 %*******************************************************************************
 \begin{equation}
 \pi_1^{et}(GL_{6g-6+2n}(K))\cong \widehat{GL}_{6g-6+2n}(K).
\end{equation}
%*******************************************************************************
Altogether,   the exact sequence (\ref{eq3.13})  can be written in  the form: 
%*****************************************************************************************************
\begin{equation}\label{eq3.14}
1\to \widehat{\mathbf{Z}}\to  \widehat{GL}_{6g-6+2n}(K)  \to  G_K\to 1. 
\end{equation}
%************************************************************************************************

%**************************************************************************
\subsection{Relation to  $\widehat{Mod} ~(g,n)$}
%***************************************************************************
Recall  that $\pi_1^{et}$ is a contravariant functor. Thererefore
the map  $GL_{6g-6+2n}(\mathbf{Z})\hookrightarrow GL_{6g-6+2n}(K)$
defines an inclusion of the  \'etale fundamental groups:
%****************************************************************************************
\begin{equation}\label{incl}
\pi_1^{et} (GL_{6g-6+2n}(K))\subseteq \pi_1^{et} (GL_{6g-6+2n}(\mathbf{Z})).
\end{equation}
%***************************************************************************************** 
On the other hand, we have:
%*******************************************************************
\begin{equation}
\left\{
\begin{array}{lll}
\pi_1^{et}(GL_{6g-6+2n}(K)) &\cong& \widehat{GL}_{6g-6+2n}(K)\\
\pi_1^{et}(GL_{6g-6+2n}(\mathbf{Z})) &\cong&  \widehat{GL}_{6g-6+2n}(\mathbf{Z}).
 \end{array}
\right.
\end{equation}
%*****************************************************************
Therefore inclusion (\ref{incl}) can be written in the form:
%************************************************************************************
\begin{equation}\label{last}
\widehat{GL}_{6g-6+2n}(K)\subseteq  \widehat{GL}_{6g-6+2n}(\mathbf{Z}). 
\end{equation}
%***********************************************************************************
But theorem \ref{thm1.1} says that  $\widehat{GL}_{6g-6+2n}(\mathbf{Z})
\cong \widehat{Mod} ~(g,n)$. Thus (\ref{last}) defines an inclusion of the profinite
groups: 
%*****************************************************************************************************
\begin{equation}\label{eq3.15}
\widehat{GL}_{6g-6+2n}(K)\subseteq \widehat{Mod} ~(g,n). 
\end{equation}
%************************************************************************************************

%**************************************************************************
\subsection{Normal subgroups of  $Mod~(g,n)$ and $G_K$}
%***************************************************************************
Recall that each closed  subgroup of  $\widehat{Mod} ~(g,n)$  is the profinite 
completion of  a normal subgroup $N$  of  the mapping class group $Mod~(g,n)$.
We conclude from (\ref{eq3.15})  that there exists an $N\unlhd Mod~(g,n)$,
such that
%****************************************************************************** 
\begin{equation}\label{main}
\widehat{N}\cong \widehat{GL}_{6g-6+2n}(K). 
\end{equation}
%*****************************************************************************
Using (\ref{main}) we can write  the exact sequence (\ref{eq3.14})
in the form:
%*****************************************************************************************************
\begin{equation}\label{eq3.16}
1\to \widehat{\mathbf{Z}}\to  \widehat{N} \to  G_K\to 1.  
\end{equation}
%************************************************************************************************
One gets from (\ref{eq3.16}) the required isomorphism:
%*****************************************************************************************************
\begin{equation}\label{eq3.161}
G_K\cong  \widehat{N} / \widehat{\mathbf{Z}}. 
\end{equation}
%************************************************************************************************
In view of  the rigidity of  $G_K$ (Section 2.4.1),  the isomorphism (\ref{eq3.161}) 
defines  a map $F_{g,n}$  from the category $\mathscr{N}_{g,n}$ to the category  $\mathscr{K}$.

\smallskip
%************************************************************************************************
\begin{remark}\label{rmk3.5}
Notice that (\ref{eq3.15}) is an isomorphism $\widehat{GL}_{6g-6+2n}(K)\cong \widehat{Mod} ~(g,n)$
if and only if $K\cong\mathbf{Q}$.  It follows from (\ref{eq3.14}), that the short exact sequence (\ref{eq3.16})
in this case corresponds to $N\cong Mod~(g,n)$ and can be written in the form:
%*****************************************************************************************************
\begin{equation}\label{eq3.165}
1\to \widehat{\mathbf{Z}}\to   \widehat{Mod} ~(g,n) \to  G_{\mathbf{Q}}\to 1.  
\end{equation}
%************************************************************************************************
\end{remark}
%*********************************************************************************************

%**************************************************************************
\section{Proofs}
%***************************************************************************
%**************************************************************************
\subsection{Proof of theorem \ref{thm1.1}}
%***************************************************************************
For the sake of clarity, let us outline the main ideas. 
Observe that $Mod ~(g,n)$ cannot be a normal subgroup  of $GL_{6g-6+2n}(\mathbf{Z})$,
since in this case  the Margulis Rigidity Theorem implies  that  $Mod ~(g,n)$ is an arithmetic group. 
We prove that $Mod ~(g,n)$ is a Zariski dense infinite index  subgroup of  $GL_{6g-6+2n}(\mathbf{Z})$. 
Following  [Venkataramana 1987] \cite[Proposition 2.1]{Ven1},  we conclude that there exists an
integer $m>0$, such that 
  %****************************************************************************************
\begin{equation}
GL_{6g-6+2n}(m\mathbf{Z})\subset Mod ~(g,n),
 \end{equation}
 %************************************************************************************
where $GL_{6g-6+2n}(m\mathbf{Z})$ is a principal congruence subgroup of $GL_{6g-6+2n}(\mathbf{Z})$
of level $m$. 
Denote by $Mod_m (g,n)$ the congruence subgroup of  $Mod ~(g,n)$
 of level $m$  [Farb \& Margalit 2011]  \cite[Section 6.4.2]{FM}.
 % Let $\mathscr{N}_m$ be a normalizer of 
 %$GL_{6g-6+2n}(m\mathbf{Z})$  in the group  $Mod ~(g,n)$,
 %i.e. $\mathscr{N}_m:=N_{Mod ~(g,n)}(GL_{6g-6+2n}(m\mathbf{Z}))$.
 We prove that:
   %****************************************************************************************
\begin{equation}\label{iso}
%\mathscr{N}_m \cong 
Mod_m (g,n)\cong GL_{6g-6+2n}(m\mathbf{Z}). 
 \end{equation}
 %************************************************************************************
The required isomorphism $\widehat{Mod} ~(g,n)\cong \widehat{GL}_{6g-6+2n}(\mathbf{Z})$
follows from  (\ref{iso}) and formula (\ref{eq1.1}).
We pass to a detailed argument by splitting the proof in a series of lemmas.

%***********************************************************************
\begin{lemma}\label{lm3.0}
The mapping class group $Mod ~(g,n)$ is a Zariski dense infinite index  subgroup of  the
arithmetic group $GL_{6g-6+2n}(\mathbf{Z})$.
\end{lemma}
%*************************************************************************
\begin{proof}
(i) Let us show that $Mod ~(g,n)$ is an infinite index subgroup of $GL_{6g-6+2n}(\mathbf{Z})$.
Assume to the contrary, that $Mod ~(g,n)$ has a finite index.
In view of the Congruence Subgroup Theorem [Bass, Lazard \& Serre 1964] \cite{BaLaSe1},  
$Mod~(g,n)$ must be a congruence subgroup
of the group $GL_{6g-6+2n}(\mathbf{Z})$.  In particular,  $Mod~(g,n)$ is an arithmetic 
group. But this is impossible, since it contains the Torelli group,  which is known to be 
an infinite index normal subgroup of the $Mod~(g,n)$. The latter contradicts the Margulis 
Rigidity, see Section 1. Thus the index  $[GL_{6g-6+2n}(\mathbf{Z}) : Mod ~(g,n)]=\infty$. 

\medskip
(ii)  Let us show that $Mod ~(g,n)$ is a Zariski dense  subgroup of $GL_{6g-6+2n}(\mathbf{Z})$.
Indeed, recall that the Tits alternative  says that  $GL_{6g-6+2n}(\mathbf{Z})$ contains 
a Zariski open solvable  subgroup or a Zariski dense free subgroup of finite rank 
[Breuillard \& Gelander 2007] \cite[Theorem 1.1]{BreGel1}. 
But $Mod ~(g,n)$ contains a free subgroup $F_r$ of finite rank, see  item (ii) of Theorem \ref{Tits}.
Thus $F_r$ is Zariski dense in   $GL_{6g-6+2n}(\mathbf{Z})$.  We conclude that the mapping class group
$Mod ~(g,n)\supset F_r$ is also Zariski dense in the arithmetic group $GL_{6g-6+2n}(\mathbf{Z})$.
\end{proof}
%**************************************************************************
%***********************************************************************
\begin{lemma}\label{lm3.1}
{\bf ([Venkataramana 1987] \cite{Ven1})}
 For an integer $m>0$ there exists a principal congruence subgroup
  $GL_{6g-6+2n}(m\mathbf{Z}) ~\unlhd
 ~GL_{6g-6+2n}(\mathbf{Z})$,  such that $GL_{6g-6+2n}(m\mathbf{Z})\subset Mod ~(g,n)$.
\end{lemma}
%*************************************************************************
\begin{proof}
The proof is an adaption of  the argument of [Venkataramana 1987] \cite[Proposition 1.4]{Ven1}
to the case of the Zariski dense subgroup  $Mod ~(g,n)$ of the linear  algebraic group
$GL_{6g-6+2n}(\mathbf{Z})$.  The details are left to the reader.
\end{proof}
%**************************************************************************

%***********************************************************************
\begin{lemma}\label{lm3.2}
$Mod_m (g,n)\cong GL_{6g-6+2n}(m\mathbf{Z})$.
\end{lemma}
%*************************************************************************
\begin{proof}
(i) The inclusion $Mod_m (g,n)\subseteq GL_{6g-6+2n}(m\mathbf{Z})$
is obvious,  since it follows from the inclusion 
$Mod ~(g,n)\subset GL_{6g-6+2n}(\mathbf{Z})$ being restricted 
to the principal congruence subgroup   $GL_{6g-6+2n}(m\mathbf{Z})$.

\medskip
(ii) Let us show that $Mod_m (g,n)\not\subset GL_{6g-6+2n}(m\mathbf{Z})$.
Indeed, let us assume to the contrary that $Mod_m (g,n)\subset GL_{6g-6+2n}(m\mathbf{Z})$.
Recall that $Mod_m (g,n)$ is a finite index subgroup of  $Mod ~(g,n)$ 
[Farb \& Margalit 2011]  \cite[Section 6.4.2]{FM}.  It is easy to see, that 
$Mod_m (g,n)$ is the maximal subgroup of $Mod ~(g,n)$  of given index.
Since the index depends only on the integer $m$, one concludes that
condition $Mod_m (g,n)\subset GL_{6g-6+2n}(m\mathbf{Z})$
contradicts the maximum principle. Therefore one gets the 
non-inclusion condition 
$Mod_m (g,n)\not\subset GL_{6g-6+2n}(m\mathbf{Z})$.

\medskip 
Lemma \ref{lm3.2} follows from items (i) and (ii). 
 \end{proof}
%**************************************************************************
%***********************************************************************
\begin{lemma}\label{lm3.3}
 $\widehat{Mod} ~(g,n)\cong \widehat{GL}_{6g-6+2n}(\mathbf{Z}).$
\end{lemma}
%*************************************************************************
\begin{proof}
It follows from lemma \ref{lm3.2}  that
%******************************************************************
\begin{equation}
\widehat{Mod}_m (g,n)\equiv \widehat{GL}_{6g-6+2n}(m\mathbf{Z}).
\end{equation}
%******************************************************************
In other words,  the inductive limits (\ref{eq1.1}) 
%*******************************************************************
\begin{equation}
\left\{
\begin{array}{lll}
\widehat{Mod} ~(g,n) &=& \varprojlim \widehat{Mod} ~(g,n) ~/ ~N_i\\
&&\\
 \widehat{GL}_{6g-6+2n}(\mathbf{Z}) &=&  \varprojlim \widehat{GL}_{6g-6+2n}(\mathbf{Z}) ~/ ~N_j
 \end{array}
\right.
\end{equation}
%*****************************************************************
coincide everywhere, except for a finite number of normal finite index
subgroups $N_i$ and $N_j$.  But such a relation means that the corresponding
profinite groups are homeomorphic, i.e.  $\widehat{Mod} ~(g,n)\cong \widehat{GL}_{6g-6+2n}(\mathbf{Z}).$
\end{proof}
%**************************************************************************

\bigskip
Theorem \ref{thm1.1} follows from lemma \ref{lm3.3}.  

%********************************************************************************
\begin{remark}
Theorem \ref{thm1.1} can be proved immediately from lemma \ref{lm3.0} and   known 
facts  about  the thin groups, see e.g. [Sarnak 2014] \cite{Sar1}.  Indeed, lemma \ref{lm3.0}
says that  $Mod ~(g,n)$ is a thin subgroup of the matrix  group $GL_{6g-6+2n}(\mathbf{Z})$.
Thus there exists an integer $q_0$,  such that for all $q$ coprime with $q_0$ the reduction
modulo $q$ map $\pi_q: Mod ~(g,n)\to GL_{6g-6+2n}(\mathbf{Z}/ q\mathbf{Z})$
is surjective  [Sarnak 2014] \cite[Section 1]{Sar1}.
In view of the fact that the map $\tau_q: GL_{6g-6+2n}(\mathbf{Z})\to GL_{6g-6+2n}(\mathbf{Z}/ q\mathbf{Z})$
is surjective for all $q\ge 1$, we conclude that   $\widehat{GL}_{6g-6+2n}(\mathbf{Z})=\varprojlim GL_{6g-6+2n}(\mathbf{Z}/ q\mathbf{Z})$
coincides with  $\widehat{Mod} ~(g,n)$ starting from some finite value $q_0$.  In other words, 
there exists an homeomorphism between the profinite groups $\widehat{Mod} ~(g,n)\cong \widehat{GL}_{6g-6+2n}(\mathbf{Z})$. 
\end{remark}
%******************************************************************************

%**************************************************************************
\subsection{Proof of theorem \ref{thm1.2}}
%***************************************************************************
\begin{proof}
(i)  To prove that the map $F_{g,n}:  N\mapsto G_K\cong  \widehat{N}/ \widehat{\mathbf{Z}}$
is a functor, we recall that the Galois groups $G_K\cong G_{K'}$,  if and only if,  $K\cong K'$
(Section 2.4.1).   
Likewise, if $N\cong N'$ are isomorphic subgroups,  then $\widehat{N}\cong \widehat{N}'$. 
Thus $F_{g,n}: \mathscr{N}_{g,n}\to \mathscr{K} $ is a functor. 

\bigskip
(ii)  If $N\not\cong N'$ is a Grothendieck pair, then $\widehat{N}\cong \widehat{N}'$.
In this case, we have $F_{g,n}(N)=F_{g,n}(N')$. It easy to see,  that  if $F_{g,n}(N)=F_{g,n}(N')$
then $N\not\cong N'$ is a Grothendieck pair. 
 In other words, the functor $F_{g,n}$  is injective everywhere except for the Grothendieck pairs.

 \bigskip
 (iii) Finally, let us prove the isomorphism (\ref{eq1.5}). 
 Let $K\in\mathscr{K}$ and denote by $K'$ a Galois extension of $K$,
 such that
%*****************************************************************************************************
\begin{equation}\label{eq3.17}
K\subseteq K'\subset \bar K. 
\end{equation}
%************************************************************************************************
 Using the results of  [Krull 1928] \cite{Kru1}, we conclude that there exists a closed finite index
 normal subgroup $G_{K'}$ of the  group $G_K$, such that
 %*****************************************************************************************************
\begin{equation}\label{eq3.18}
Gal~(K' | K)\cong G_K / G_{K'},
\end{equation}
%************************************************************************************************
 where $G_{K'}$ is  the absolute Galois group of the number field $K'$. 
 From (\ref{eq3.16}) we get 
 %*****************************************************************************************************
\begin{equation}\label{eq3.19}
G_K\cong \widehat{N} / \widehat{\mathbf{Z}},
\end{equation}
%************************************************************************************************
 where $N \unlhd Mod~(g,n)$.  Since $K'\in\mathscr{K}$, there exists $N'\in \mathscr{N}_{g,n}$,
 such that $K'=F_{g,n}(N')$.  Moreover, because $K\subseteq K'$, one gets an inclusion 
 $N'\subseteq N$, where $N'$ is a finite index normal subgroup of $N$.   Since $G_{K'}\subseteq G_K$,  the groups $N$ and $N'$ are not a Grothendieck 
 pair,  unless $N'\cong N$.   Therefore one gets  from (\ref{eq3.16}): 
 %*****************************************************************************************************
\begin{equation}\label{eq3.20}
G_{K'}\cong \widehat{N}' / \widehat{\mathbf{Z}}.
\end{equation}
%************************************************************************************************
 We can substitute (\ref{eq3.19}) and (\ref{eq3.20}) into the formula (\ref{eq3.18}):
%*****************************************************************************************************
\begin{equation}\label{eq3.21}
Gal~(K' | K)\cong \left(\widehat{N} / \widehat{\mathbf{Z}}\right)  \big/
 \left(\widehat{N}' / \widehat{\mathbf{Z}}\right)\cong
 \widehat{N} / \widehat{N}'\cong 
 \widehat{N/N'}. 
\end{equation}
%************************************************************************************************ 
 But $N/N'$ is a finite group and therefore $\widehat{N/N'}\cong N/N'$. 
 Thus formulas (\ref{eq3.21}) imply that
 %*****************************************************************************************************
\begin{equation}\label{eq3.22}
Gal~(K' | K)\cong N / N', \quad\hbox{where}  ~N' \unlhd N.
\end{equation}
%************************************************************************************************ 

 \bigskip
 Theorem \ref{thm1.2} is proven. 
 \end{proof}

%**************************************************************************
\subsection{Proof of corollary \ref{cor1.3}}
%***************************************************************************
\begin{proof}
The proof is a straightforward application of the Tits alternative for $Mod~(g,n)$,
see  [McCarthy 1985]  \cite[Theorem A]{Mac1} or Section 2.1. 
Indeed, consider a group $N\unlhd Mod~(g,n)$.   The Tits alternative says that there exists 
a subgroup  $N'\unlhd N$, such that:

\medskip
(i) either $[N:N']<\infty $ and $N'$ is free abelian group of the maximal rank $3g-3+n$, 

\smallskip
(ii) or  $N'$ is free non-abelian group, i.e. $N'\unlhd F_2$. 

\bigskip
Denote by $K$ and $K'$ the number fields, such that $K=F_{g,n}(N)$ and 
$K'=F_{g,n}(N')$.  Since $N'\unlhd N$,  theorem \ref{thm1.2} says that:
 %*****************************************************************************************************
\begin{equation}\label{eq3.23}
K\subseteq K'\subset \bar K. 
\end{equation}
%************************************************************************************************ 

To calculate the absolute Galois group $G_{K'}$,  we must consider the following alternative  cases.

\bigskip
(i) Let $N'$ be  a free abelian group of the rank $r\le 3g-3+n$. 
It is well known,  that each finite index subgroup $N''\unlhd N'$
can be  found  from  the short exact sequence:
 %*****************************************************************************************************
\begin{equation}\label{eq3.24}
0\to  N'' \buildrel\rm A \over\longrightarrow N' \to
\mathbf{Z}/k_1\mathbf{Z}\oplus\dots\oplus \mathbf{Z}/k_r\mathbf{Z}
\to 0,
\end{equation}
%************************************************************************************************ 
 where the integers $(k_1 | k_2 \dots | k_r)$ are defined by the Smith normal form of the 
 matrix $A\in GL_r(\mathbf{Z})$.   In particular, if $K''=F_{g,n}(N'')$ is an extension 
 of the field $K'$ corresponding to the subgroup  $N''\unlhd N'$,  then  formula (\ref{eq3.22}) implies
  that
 %*****************************************************************************************************
\begin{equation}\label{eq3.25}
Gal~(K''|K')\cong N'/N''\cong
\mathbf{Z}/k_1\mathbf{Z}\oplus\dots\oplus \mathbf{Z}/k_r\mathbf{Z}.
\end{equation}
%************************************************************************************************ 
 In other words, the absolute Galois group $G_{K'}$ is a profinite completion of the free abelian group $\mathbf{Z}^r$, 
 i.e. 
 %*****************************************************************************************************
\begin{equation}\label{eq3.26}
G_{K'}\cong \widehat{\mathbf{Z}}^r.
\end{equation}
%************************************************************************************************ 

\bigskip
(ii)  Let $N'$ be a free non-abelian group, i.e. $N'\unlhd F_2$,
where $F_2$ is the free group on two generators. 
It is well known,  that each finite index subgroup $N''\unlhd N'$
can be  found  from  the short exact sequence:
 %*****************************************************************************************************
\begin{equation}\label{eq3.27}
0\to  N'' \to N' \to
G
\to 0,
\end{equation}
%************************************************************************************************ 
 where $G$ is a finite  group of order $k$.  
 The rank  of the free group $N''$ is given by the famous  Nielsen-Schreier formula
  $r''=1+k(r'-1)$,  where $r'$ is the rank of $N'$.  
If $K''=F_{g,n}(N'')$ is an extension 
 of the field $K'$ corresponding to the subgroup  $N''\unlhd N'$,  then  formula (\ref{eq3.22}) implies
  that
 %*****************************************************************************************************
\begin{equation}\label{eq3.28}
Gal~(K''|K')\cong 
 N'/N''\cong G.
\end{equation}
%************************************************************************************************ 
In other words,  the absolute Galois group $G_{K'}$ is a profinite completion of the free  group $F_{r'}$
of rank $r'$,  i.e. 
%*****************************************************************************************************
\begin{equation}\label{eq3.29}
G_{K'}\cong \widehat{F}_{r'},  \quad\hbox{where} ~r'\ge 2. 
\end{equation}
%************************************************************************************************

Corollary \ref{cor1.3} follows. 
\end{proof}

%**************************************************************************
%\subsection{Proof of remark \ref{rmk1.3}}
%***************************************************************************
%\begin{proof}
%It is well known, that the Galois group of an abelian extension  is given by the
%formula (\ref{eq3.25}),  while the Galois group of a cyclotomic extension is 
%given by the formula (\ref{eq3.28}). Clearly,  both  are the  abelian extensions.
%\end{proof}

%**************************************************************************
\subsection{Proof of corollary \ref{cor1.4}}
%***************************************************************************
\begin{proof}
The Shafarevich conjecture can be derived from
following lemma. 
%*****************************************************************************
\begin{lemma}\label{lm3.6}
The mapping class group of every orientable surface $X$ of genus $g$ with $n$ boundary
components contains a free non-abelian subgroup of countable rank, i.e.
%*****************************************************************************************************
\begin{equation}\label{eq3.30}
F_{\infty}\subset Mod~(g,n). 
\end{equation}
%************************************************************************************************
\end{lemma}
%*****************************************************************************
\begin{proof}
We refer the reader to Section 2.1 for the notation and definitions. 
It is well known, that any collection $\mathfrak{R}$ of pseudo-Anosov mapping classes
with pairwise distinct measured foliations $\left(\mathcal{F}_s^{(i)}, \mathcal{F}_u^{(i)}\right)$
generates a free non-abelian subgroup $F_r$ of rank $r=|\mathfrak{R}|$
of the  mapping class group $Mod~(g,n)$,  provided each element is first raised to a 
sufficiently high power, see e.g.  [McCarthy 1985]  \cite{Mac1}.

On the other hand, for every orientable surface $X$ there exists a countable set of the 
pairwise distinct measured foliations $\left(\mathcal{F}_s^{(i)}, \mathcal{F}_u^{(i)}\right)$
\cite[Section 0.3.2]{N}.  In other words, for every surface $X$ there exists a collection 
$\mathfrak{R}$ of the pseudo-Anosov mapping classes, such that  
%*****************************************************************************************************
\begin{equation}\label{eq3.31}
r=|\mathfrak{R}|=\infty. 
\end{equation}
%************************************************************************************************
Such a collection $\mathfrak{R}$ generates a free non-abelian subgroup $F_{\infty}$ 
of the group $Mod~(g,n)$. Lemma \ref{lm3.6} is proven.
\end{proof}

\bigskip
Let us return to the proof of corollary \ref{cor1.4}. 
In view of the remark \ref{rmk3.5}, the case $K\cong\mathbf{Q}$
corresponds to the improper subgroup $N\cong Mod~(g,n)$ 
of the group $Mod~(g,n)$.  In view of  lemma \ref{lm3.6},  one gets: 
%*****************************************************************************************************
%\begin{equation}\label{eq3.32}
$F_{\infty}\subset N\cong Mod~(g,n)$. 
%\end{equation}
%************************************************************************************************
We  substitute $N=F_{\infty}$ into the exact sequence  (\ref{eq3.16}).  
One  gets:
%*****************************************************************************************************
\begin{equation}\label{eq3.33}
G_K\cong \widehat{F}_{\infty}/\widehat{\mathbf{Z}}\cong 
\widehat{\left(F_{\infty}/\mathbf{Z}\right)}\cong
\widehat{\left(F_{\infty}/F_1\right)}\cong  \widehat{F}_{\infty}, 
\end{equation}
%************************************************************************************************
where an isomorphism $\mathbf{Z}\cong F_1$ has been used.

It remains to show,  that 
in (\ref{eq3.33}) we have  $K\cong \mathbf{Q}^{ab}$.
Consider the inclusions of groups $F_1\subset F_{\infty}\subset F_2$. 
In view of theorem \ref{thm1.2},  one gets an inclusion of the number fields
%******************************************************************************
\begin{equation}\label{eq3.34}
K'\subseteq K\subseteq K'',
\end{equation}
%******************************************************************************
  where $K'=F_{g,n}(F_2)$ and $K''=F_{g,n}(F_1)$. 
It is easy to see, that $\mathbf{Q}^{ab}\subseteq K'$  and $K''\cong \mathbf{Q}^{ab}$.  
Indeed, $K''\cong \mathbf{Q}^{ab}$ because $\widehat{F}_1\cong\widehat{\mathbf{Z}}\cong Gal~(Q^{ab}|Q)$. 
On the other hand, the short exact sequence 
%*****************************************************************************************************
\begin{equation}\label{eq3.35}
0\to  N' \to F_2\to
\left(\mathbf{Z}/k\mathbf{Z}\right)^{\times}
\to 0
\end{equation}
%************************************************************************************************ 
implies that $\mathbf{Q}^{ab}\subseteq K'$, because  the subgroup $N'\subset F_2$ corresponds to 
a cyclotomic  extension $F_{g,n}(N')$ of $\mathbf{Q}$. Thus from  (\ref{eq3.34}) we have 
$Q^{ab}\subseteq K\subseteq Q^{ab}$.  We conclude  
that  $K\cong \mathbf{Q}^{ab}$. 
In other words,  $G_{\mathbf{Q}^{ab}}\cong \widehat{F}_{\infty}$.
Corollary \ref{cor1.4} is proven. 
 \end{proof}

\bibliographystyle{amsplain}

%**********************************************************

\end{document}